\documentclass[a4paper, 12pt]{amsart}
\usepackage[margin=1.9cm]{geometry}
\usepackage[english]{babel}
\usepackage{amsmath}
\usepackage{amsfonts}
\usepackage{amssymb}
\usepackage{amsthm}
\usepackage{xurl}
\usepackage[T1]{fontenc}
\usepackage{color}
\usepackage[dvipsnames]{xcolor}
\usepackage{tikz}
\usepackage{tabularx}
\usepackage[font=scriptsize]{caption}
\usepackage{mathrsfs}
\usepackage{hyperref}
\usepackage{cleveref}
\allowdisplaybreaks

\newtheoremstyle{one}
{11pt}
{11pt}
{\it}
{}
{\bf}
{.}
{1mm}
{}

\newtheoremstyle{two}
{11pt}
{11pt}
{}
{}
{\bf}
{.}
{1mm}
{}

\theoremstyle{one}
\newtheorem{theorem}{Theorem}[section]
\newtheorem{question}[theorem]{Question}
\newtheorem{lemma}[theorem]{Lemma}

\newtheorem{proposition}[theorem]{Proposition}
\newtheorem{corollary}[theorem]{Corollary}

\theoremstyle{two}
\newtheorem{definition}[theorem]{Definition}
\newtheorem{example}[theorem]{Example}
\newtheorem{alg}[theorem]{Algorithm}

\newtheorem{remark}[theorem]{Remark}

\title{Generalized Eckardt Points on del Pezzo surfaces of degree 1}
\author{Julie Desjardins}
\author{Yu Fu}
\author{Kelly Isham}
\author{Rosa Winter}

\begin{document}

\begin{abstract}
    We study intersections of exceptional curves on del Pezzo surfaces of degree 1, motivated by questions in arithmetic geometry. Outside characteristics 2 and 3, at most 10 exceptional curves can intersect in a point. We classify the different ways in which 10 exceptional curves can intersect, construct a new family of surfaces with 10 exceptional curves intersecting in a point, and discuss strategies for finding more such examples.
\end{abstract}

\maketitle

\section{Introduction}
Del Pezzo surfaces are smooth projective algebraic surfaces with ample anticanonical divisor, and can be classified by their \textsl{degree}, which is defined as the self-intersection of the canonical divisor; an integer between 1 and 9. Both the geometric as well as the arithmetic complexity of del Pezzo surfaces goes up as the degree goes down (see also the \textsl{Motivation from the rational points perspective} below).
Over an algebraically closed field, a del Pezzo surface contains a fixed number of exceptional curves (which we often call \textsl{lines}), depending on the degree; the lower the degree, the more lines on the surface. 

Del Pezzo surfaces of degree 1 contain 240 lines over an algebraically closed field, and at most 10 lines intersect in one point outside characteristic 2 and 3 (at most 16 and 12 in those characteristics, respectively) \cite{vLW}. Besides this, little is known about the configurations of these lines. This is in sharp contrast with del Pezzo surfaces of higher degree. For example, del Pezzo surfaces of degree 3 (the smooth cubic surfaces in $\mathbb{P}^3$) contain 27 lines, and any point on the surface is contained in at most 3 lines; such a point is called an \textsl{Eckardt point}. There are at most 45 Eckardt points on a cubic surface in characteristic 2, and at most 18 otherwise. A del Pezzo surface of degree 2 contains 56 lines, a point on it is contained in at most 4 lines (a \textsl{generalized Eckardt point}), and there are at most 126 generalized Eckardt points on such a surface. 

Following the terminology established for del Pezzo surfaces of degrees 2 and 3, we call a point on a del Pezzo surface of degree 1 over a field with characteristic unequal to 2 and 3, that is contained in the intersection of 10 lines, a \textsl{generalized Eckardt point}. For characteristics 2 and 3 we define this to be the intersection point of 16 and 12 lines, respectively. The knowledge we have about del Pezzo surfaces of degree at least 2 prompts the following questions.

\begin{question}\label{Q1}
    How many generalized Eckardt points can a del Pezzo surface of degree 1 contain? 
\end{question}

\begin{question}\label{Q2}
    How likely is it for a del Pezzo surface of degree 1 to contain a generalized Eckardt point?
\end{question}

It is worth mentioning that there are only a handful of examples in the literature of (families of) del Pezzo surfaces of degree 1 with a generalized Eckardt point. Moreover, all these examples show only 1 such point on the surface \cite{SVL,vLW}.

\vspace{11pt}

In this paper we start an investigation towards answering Questions \ref{Q1} and \ref{Q2}. Our main contribution is the following result (c.f. Theorem \ref{thm: 9 orbits}):

\begin{theorem}\label{thm: intro 9 orbits}
Outside characteristic 2, there are 9 possible configurations for 10 exceptional curves on a del Pezzo surface of degree 1 to intersect in one point, and they are listed in Figures \ref{table cliques} and \ref{table clique 9}. In characteristic 0, at most 7 of the configurations (1, 2, 3, 5, 6, 8 in Figure \ref{table cliques} and the configuration in Figure \ref{table clique 9}) can correspond to a generalized Eckardt point on a del Pezzo surface of degree 1. 
\end{theorem}

Theorem \ref{thm: 9 orbits} is a first step in answering Questions \ref{Q1} and \ref{Q2}, as it tells us which configurations of intersecting lines to count. As mentioned, only two examples of (families of) del Pezzo surfaces of degree 1 with a generalized Eckardt points can be found in the current literature. In \cite[Example 4.1]{SVL}, the authors construct a family of del Pezzo surfaces of degree 1 that all contain 10 exceptional curves that intersect in one point, whose intersection graph is isomorphic to number 5 in Figure \ref{table cliques}. In \cite[Example 5.2]{vLW} the authors construct a del Pezzo surface of degree 1 over a field of characteristic unequal to 2 with 10 exceptional curves that intersect in one point, whose intersection graph is isomorphic to the one in Figure \ref{table clique 9}. We extend this to a whole family of such examples in Theorem \ref{thm:family_on_ram_curve}. 

The strategy used in order to obtain the examples of Theorem \ref{thm:family_on_ram_curve} is explained in detail with Algorithm \ref{alg:strategy}. This strategy combined with combinatorial computations proves that there are no del Pezzo surface of degree 1 with a generalized Eckardt point corresponding to the configurations in Figure \ref{table cliques} over the finite field $\mathbb{F}_p$ with $p\in\{1,\ldots,17,23\}$ (Theorems \ref{thm:finitefield17} and \ref{thm: no Eckardt points primes 17,23}) and allows us to find an example of a del Pezzo surface of degree 1 with a generalized Eckardt point corresponding to configuration 8 over $\mathbb{F}_{19}$ (Theorem \ref{thm: clique 8 in F19}). We are quite confident that our strategy will lead to more results of the kind we are presenting in Sections 5 and 6.

\vspace{11pt}

\textbf{Motivation from the rational points perspective.} One motivation to study Questions~\ref{Q1} and \ref{Q2} comes from the arithmetic of del Pezzo surfaces. Over algebraically closed fields, these surfaces are \textsl{rational} (i.e., birationally equivalent to the projective plane). Over non-algebraically closed fields, this is not true in general, and the set $X(k)$ of rational points on a del Pezzo surface $X$ over such a field $k$ is far from fully understood. It is generally believed that del Pezzo surfaces that contain a rational point have `many' rational points, and that these rational points are `well distributed'. There are several results that support this belief. Del Pezzo surfaces of degree at least 3 with a rational point are \textsl{unirational} (i.e., dominated by the projective plane) \cite{Se43,Se51,Manin,Ko02,Pie}, for del Pezzo surfaces of degree 2 the same holds if they contain a rational point outside a closed subset \cite{STVA}, and for del Pezzo surfaces of degree 1 unirationality is known when the surface contains a conic bundle \cite{KM17}. If the field $k$ is infinite, then unirationality over $k$ implies that the set $X(k)$ is \textsl{Zariski dense} (i.e., not contained in a closed subset). There are several partial results on Zariski density of $X(k)$ for a del Pezzo surface of degree 1 \cite{Ulas,Ulas2,UT10, Jabara,SVL,Desjardins1,DW}. Another measure of `many' and `well distributed' rational points for a variety $X$ over a number field $k$ is \textsl{weak approximation} (i.e., the set $X(k)$ of rational points is dense in the product of local points $\prod_{v\in\Omega_k}X(k_v)$, where $\Omega_k$ is the set of places of $k$). Weak approximation is satisfied for all del Pezzo surfaces of degree at least 5 with a rational point, and shown not to hold in general for del Pezzo surfaces of lower degree. However, it is shown that del Pezzo surfaces of degree 3 and 4 with a rational point satisfy \textsl{weak weak approximation} (i.e., $X(k)$ is dense in $\prod_{v\in\Omega_k\setminus S}X(k_v)$ for a finite set $S$) \cite{SS,SD}, and the same holds for del Pezzo surfaces of degree 2 with a general enough rational point \cite{DSW} and del Pezzo surfaces of degree 1 with a conic bundle under an extra condition \cite{DS}. 

The configuration of lines often come up when proving these arithmetic results. For example, the results in \cite{STVA,DSW} require the existence of a rational point that is not contained in the intersection of 4 lines on a del Pezzo surface of degree 2. For del Pezzo surfaces of degree~1, the result in \cite{DS} requires that no 4 of the exceptional curves over the algebraic closure intersect in one point, the result in \cite{SVL} requires the existence of a point not contained in 6 or more lines, and in \cite{DW} one needs a point which is not torsion on its fiber of a corresponding elliptic surface, which implies that it can not be contained in 9 or more lines, as was shown in~\cite{DW2}.
\vspace{11pt}

\textbf{Outline of the paper}\color{black}

In Section \ref{sec:background} we give the necessary background. In Section \ref{sec:configurations} we determine the different configurations in which 10 exceptional curves on a del Pezzo surface of degree 1 can pairwise intersect (Theorem \ref{thm: 9 orbits} and Remark \ref{rem: on the ramification curve}). Section \ref{sec:strategies} contains a discussion of how to determine which of these configurations can correspond to 10 exceptional curves that all intersect in \textsl{one point}, i.e., which configurations can correspond to generalized Eckardt points (Algorithm \ref{alg:strategy}). In Section \ref{sec:families} we construct a family of surfaces that all contain a generalized Eckardt point (Theorem \ref{thm:family_on_ram_curve}), and we study the same questions over finite fields in Section \ref{sec:finite fields} (Theorems \ref{thm:finitefield17}, \ref{thm: no Eckardt points primes 17,23}, \ref{thm: clique 8 in F19}). We end with a short discussion on open questions in Section \ref{sec: future}.

\vspace{11pt}

Computations were done in \texttt{magma} \cite{magma}. The code can be found at \cite{DFIWcode}. 

\vspace{11pt}

\textbf{Acknowledgements.} This project originated during the workshop `Women in Numbers 6', held in March 2023 at BIRS in Banff, Canada. We are grateful for the organisers and the organisation of BIRS. We thank Bianca Viray for useful discussions. We are grateful to the anonymous referee for useful comments that improved the quality of the paper. Julie Desjardins was partially supported by an NSERC Discovery grant. Rosa Winter was supported by UKRI fellowship MR/T041609/2.

\section{Background}\label{sec:background}
In this section, we give the necessary background for the rest of this paper, namely the geometric properties of del Pezzo surfaces of degree 1, and the weighted graph on exceptional classes. We show how this graph and its automorphism group come into play when studying intersecting exceptional curves.

\subsection{Del Pezzo surfaces of degree one and exceptional curves}

\begin{definition}
Let $1 \le n \le 8$ be an integer and let $P_1, \ldots, P_n$ be points in $\mathbb{P}^{2}$. We say that $P_1, \ldots, P_n$ are in \textsl{general position} if there is no line containing three of the points, no conic containing six of the points, and no cubic containing eight of the points with a singularity at one of them. 
\end{definition}

A del Pezzo surface over a field $k$ is a smooth, projective, geometrically integral surface with ample anticanonical divisor. The degree of a del Pezzo surface is the self-intersection number of the canonical divisor, which is an integer between 1 and 9. Moreover, a del Pezzo surface over an algebraically closed field is either isomorphic to $\mathbb{P}^{1} \times \mathbb{P}^{1}$ when $d=8$ or to $\mathbb{P}^{2}$ blown up at $9-d$ points in general position otherwise \cite[Theorem 24.4]{Manin}.
Let $X$ be a del Pezzo surface of degree 1 over an algebraically closed field $k$, then $X$ is isomorphic to the blow up of $\mathbb{P}^{2}$ in eight points $P_1,\ldots,P_8$ in general position. Let $\operatorname{Pic}(X)$ be the Picard group of $X$. Let $K_X$ be a fixed divisor in the canonical class of $X$.\begin{definition}
Define an \textit{exceptional curve} on $X$ to be an irreducible projective curve $C \subset X$ such that 
$$C \cdot C=-1, \text{  } C \cdot K_X=-1.$$
The corresponding class $[C] \in \operatorname{Pic}(X)$ is called an \textit{exceptional class} with
$$[C] \cdot [C]=-1, \text{  } [C] \cdot [K_X]=-1.$$
\end{definition}

 For $i \in \{ 1, \ldots, 8\}$, let $[E_{i}]$ be the exceptional class of the exceptional curve above $P_i$, and $[L]$ be the class of the pullback of a line $\ell \subset \mathbb{P}^2$ not passing through any of the $P_i$. Once we realize $\operatorname{Pic}(X)$ as a lattice isomorphic to $\mathbb{Z}^9$, we can write down a basis as $\{[L], [E_1], \ldots, [E_8]\}$ \cite[Proposition 20.9.1]{Manin}. Moreover, for $i,j \in \{1, \ldots, 8\}$, $i \ne j$, we have 
 $$[E_i] \cdot [E_i]=-1, \text{  } [E_i] \cdot [E_j] =0,  \text{   } [L] \cdot [E_i]=0, \text{  } [L] \cdot [L]=1.$$ 

\begin{remark}
 We often use the word \textsl{lines} in order to speak about exceptional curves. This comes from the fact that for del Pezzo surfaces of degree $d\geq3$, the exceptional curves are lines on the image of the surface under the anticanonical embedding in $\mathbb{P}^d$. 
\end{remark}
 
  Manin proved that every exceptional class in $\operatorname{Pic}(X)$ contains exactly one exceptional curve on $X$ and gave a geometrical description of the exceptional classes. 
\begin{theorem}[{{\cite[Theorem 26.2]{Manin}}}] \label{exceptional curves}
	The exceptional curves (therefore exceptional classes) on the del Pezzo surface $X$ which is obtained by blowing up 8 points $P_1, \ldots, P_8$ in general position are 
	\begin{itemize}
		\item The exceptional curves $E_i$ above the points $P_i$, for $i \in \{1, \ldots, 8\}$,
	\end{itemize} 
	and the strict transforms of the following curves in $\mathbb{P}^2$.
	\begin{itemize}
		\item The lines $\ell_{i,j}$ passing through the points $P_{i}$ and $P_j$ for $i \ne j$;
		\item the conics passing through any five of the points;
		\item the cubics $c_{i,j}$ not passing through $P_i$, passing once through points $P_k$ with $k \ne i, j$, passing twice through $P_{j}$;
		\item the quartics $q_{i, j, k}$ passing through the eight points with a double point in $P_i, P_j$ and $P_k$ for $i, j, k$ distinct;
		\item the quintics passing through the eight points with double points at six of them;
		\item the sextics passing through the eight points with double points at seven of them, and a triple point at one of them.
	\end{itemize}
\end{theorem}

One can write the exceptional classes in the form of linear combinations of the basis $$a[L]-\sum_{i=1}^{8}b_i[E_i]$$ where $a$ is the degree of the corresponding curve in $\mathbb{P}^2$ and $b_i$ is the multiplicity of $P_i$. In this way, we obtain a vector $(a, b_1, \ldots, b_8)$ in $\mathbb{Z}^9$. Let $A$ be the set of 240 vectors obtained in this way, and let $C$ be the set of all exceptional classes in $\operatorname{Pic}(X)$. There is a one-to-one correspondence between $A$ and $C$ induced by the basis $L,E_1,\ldots,E_8$ \cite[Lemma 2.4, Remark 2.5]{vLW}. Therefore we are able to study the intersection of exceptional curves (and classes) by carrying out computations in some vector spaces, either by \texttt{magma} or by hand.  
\subsection{The weighted graph on exceptional classes}

\begin{definition}
	 Define a \textit{graph} to be a pair $(V, D)$, where $V$ is a set of elements called vertices, and $D$ a subset of the power set of $V$ such that every element in $D$ has cardinality 2; elements in $D$ are called edges, and the size of the graph is the cardinality of $V$. A graph $(V, D)$ is complete if for every two distinct vertices $v_1, v_2 \in V$, the pair $\left\{v_1, v_2\right\}$ is in $D$.

Define a \textit{weighted graph} to be a graph $(V, D)$ with a map $\psi: D \longrightarrow A$, where $A$ is any set, whose elements we call weights; for any element $d$ in $D$ we call $\psi(d)$ its weight. If $(V, D)$ is a weighted graph with weight function $\psi$, then we define a weighted subgraph of $(V, D)$ to be a graph $\left(V^{\prime}, D^{\prime}\right)$ with map $\psi^{\prime}$, where $V^{\prime}$ is a subset of $V$, while $D^{\prime}$ is a subset of the intersection of $D$ with the power set of $V^{\prime}$, and $\psi^{\prime}$ is the restriction of $\psi$ to $D^{\prime}$.

An \textsl{isomorphism} between two weighted graphs $(V, D)$ and $\left(V^{\prime}, D^{\prime}\right)$ with weight functions $\psi: D \longrightarrow$ $A$ and $\psi^{\prime}: D^{\prime} \longrightarrow A^{\prime}$, respectively, consists of a bijection $f$ between the sets $V$ and $V^{\prime}$ and a bijection $g$ between the sets $A$ and $A^{\prime}$, such that for any two vertices $v_1, v_2 \in V$, we have $\left\{v_1, v_2\right\} \in D$ with weight $w$ if and only if $\{f\left(v_1\right), f\left(v_2\right)\} \in D^{\prime}$ with weight $g(w)$. We call the map $f$ an \textsl{automorphism} of $(V, D)$ if $(V, D)=\left(V^{\prime}, D^{\prime}\right), \psi=\psi^{\prime}$, and $g$ is the identity on~$A$.
\end{definition}

\begin{definition}
A \textit{clique} of a weighted graph is a complete weighted subgraph. A \textit{maximal clique} is a maximal complete weighted subgraph with respect to the number of vertices. 
\end{definition}

Let $X$ be a del Pezzo surface of degree 1 with canonical divisor $K_X$, and let $K_X^{\perp}$ be the orthogonal complement of $K_X$ in $\operatorname{Pic}(X)$. Let $\langle \cdot , \cdot \rangle$ be the \textit{negative} of the intersection pairing on $\operatorname{Pic}(X)$ and note that $\langle \cdot , \cdot \rangle$ induces an inner product on $\mathbb{R} \otimes_{\mathbb{Z}} \operatorname{Pic}(X)$ which realizes $K_X^{\perp}$ as Euclidean space. Vectors in $K_X^{\perp}$ with length $\sqrt{2}$, i.e., classes in $K_X^{\perp}$ with self-intersection $-2$, form a root system of type $E_8$ \cite[Theorem 23.9]{Manin}:
$$E=\{D \in \operatorname{Pic} (X) \mid \langle D, D \rangle=2 ; D \cdot K_X=0\}.$$ 
There is a bijection between the set $C$ of 240 exceptional classes in $\operatorname{Pic}(X)$ and $E$:
$$F: C \longrightarrow E, \text{  } c \longmapsto c+K_X. $$
For $c_1, c_2 \in C$ we have $\langle c_1 +K_X, c_2 + K_X \rangle=1 - c_1 \cdot c_2$.
As a consequence of this bijection, the group of permutations of $C$ that preserve
the intersection pairing is isomorphic to the Weyl group $W_8$, which is the group of
permutations of $E_8$ generated by the reﬂections in the hyperplanes orthogonal to
the roots \cite[Theorem 23.9]{Manin}.

\vspace{11pt}

Denote by $\Gamma$ the complete weighted graph whose vertex set is the set of simple roots in $E$, and whose weight function is induced by the inner product on the Euclidean space. Denote by $G$ the complete weighted graph whose vertex set is $C$, and whose weight function is the intersection pairing in $\operatorname{Pic}(X)$. 
 Then there is an isomorphism $f$ between $G$ and $\Gamma$, which sends an edge $d=\{c_1, c_2\}$ in $G$ with weight $\omega$ to the edge $\tilde{d}=\{c_1 +K_X, c_2 + K_X \}$ of weight $1-\omega$ in $\Gamma$.

Note that, since the action of $W_8$ on the set $C$ preserves the intersection pairing, the group $W_8$ is also the automorphism group of $G$ (and of $\Gamma$). 

\vspace{11pt}
We aim to study the intersection of exceptional curves in $X$. If two such curves intersect, then the corresponding vertices in $G$ are connected by edges of positive weights. The possible positive weights that occur in $G$ are 1, 2, and 3, so the corresponding weights that occur in $\Gamma$ are 0, $-1$, and $-2$. Therefore, the size of maximal cliques in $\Gamma$ with edges of these weights gives an upper bound for the maximal possible number of lines on $X$ that go through one point. Note that this upper bound does not need to be sharp.

\begin{definition}
    We call a set of exceptional curves \textsl{concurrent} if there is a point on $X$ contained in all of them.
\end{definition}

  As mentioned in the introduction, the maximal number of concurrent exceptional curves on a del Pezzo surface of degree 3 is 3, and for a del Pezzo surface of degree 2, the number of exceptional curves that are concurrent in a point is at most 4. In both these cases, this number is the same as the maximal size of complete subgraphs of the intersection graph on the exceptional curves. But for del Pezzo surface of degree 1, things are more complicated. The maximal size of cliques in the graph $G$ is 16, hence this is also the maximal number of exceptional curves that are concurrent in a point on $X$ \cite[Proposition 2.8]{vLW}. However, this is only sharp in characteristic 2; for char $k \ne 2$, the number of exceptional curves that are concurrent in a point is not given by the maximal size of cliques in $G$; it is at most 10 outside characteristics 2 and 3, and at most 12 in characteristic~3 \cite[Theorems 1.1, 1.2]{vLW}. Both of these upper bounds are shown to be sharp except possibly in characteristic 5. 

  \vspace{11pt}
  
  To answer Questions \ref{Q1} and \ref{Q2}, we study maximal cliques of size ten in $\Gamma$, especially those with weights $\{-1, 0\}$ and $\{-2,0\}$, as we will explain now.
  
  \begin{definition}
  By $\Gamma_{\{-2, 0\}}$ and $\Gamma_{\{-1, 0\}}$ we denote the subgraph of $\Gamma$ whose vertices are all the vertices of $\Gamma$, and whose edges are only those in $\Gamma$ of weights $\{-2, 0\}$ and $\{-1, 0\}$ respectively. 
  \end{definition}

\begin{lemma}[{\cite[Lemma 2.7]{vLW}}]\label{intersection multiplicity}
\begin{itemize}
\item[]
	\item[(i)] Let $e$ be an exceptional class in $\operatorname{Pic}(X)$. Then there is exactly one exceptional class $f$ with $e \cdot f=3$, there are 56 exceptional classes $f$ with $e \cdot f=0$, there are 126 exceptional classes $f$ with $e \cdot f=1$, and 56 exceptional classes $f$ with $e \cdot f=2$.
\item[(ii)] For $e_1 , e_2$ two exceptional classes in $\operatorname{Pic}(X)$ with $e_1\cdot e_2 = 3$, and $f$ a third exceptional class, we have $e_1\cdot f +e_2 \cdot f  = 2$.
\end{itemize}	
\end{lemma}
 
 The bi-anticanonical divisor ${-2K_X}$ induces a morphism \begin{equation}\label{eq:map phi}
 \varphi: X \to \mathbb{P}^3
 \end{equation} such that $X$ can be realized as a double cover of a cone $Q$ in $\mathbb{P}^3$, ramified over a sextic curve \cite[Proposition 2.3]{CO99}. Moreover, by \cite[Proposition 2.6]{CO99}, $\varphi$ maps an exceptional curve on $X$ bijectively to a smooth conic which is the intersection of $Q$ with a hyperplane in $\mathbb{P}^3$ not containing the vertex of $Q$. Conversely, for a hyperplane section $H$ of $Q$ not containing the vertex of $Q$ and tri-tangent to the smooth sextic curve, the pullback $\varphi^{\ast}H = e_1 + e_2$ under $\varphi$ is the sum of two exceptional curves intersecting with multiplicity 3. We call $e_1$ the \textit{partner} of~$e_2$. Thus every point on $H$ has two preimages, except for the points with preimages in $e_1 \cap e_2$. 
 
 \begin{remark}\label{rem: partners}We conclude from the above that if a set of exceptional curves concurrent in $P$ contains an exceptional curve and its partner, then $P$ is contained in the ramification locus of $\varphi$. Conversely, if a set of exceptional curves is concurrent at a point $P$ that lies on the ramification locus, then the set of partners of the exceptional curves are also concurrent at $P$. Moreover, by Lemma \ref{intersection multiplicity}, a set of exceptional curves that is concurrent at one point on the ramification curve forms a clique in $G$ with only edges of weights 1 and 3, and a set of exceptional curves concurrent at one point outside the ramification curve forms a clique in $G$ with only edges of weights 1 and 2. 
 \end{remark}
 
 Cliques in $G$ with only edges of weights 1 and 3 or 1 and 2 correspond to cliques in $\Gamma_{\{-2, 0\}}$ and in $\Gamma_{\{-1, 0\}}$, respectively. This motivates us to study those two graphs, especially the maximal cliques in them.    
 
\section{Configurations of ten concurrent lines}\label{sec:configurations}
A naive approach to give an upper bound for the answer to Question \ref{Q1} is to count all cliques of size 10 in the graph $G$ on the 240 exceptional curves on a del Pezzo surface of degree 1. Although this does give an upper bound for the number of generalized Eckardt points on a del Pezzo surface of degree 1, it is very large (only counting those cliques with edges of weights 1 and 2 we obtain already 107,700,096 cliques of size 10, which is the sum of the orbits in Table~\ref{orbit_size_table}), and far from sharp. For example, we know that some cliques in $G$ never correspond to exceptional curves that are concurrent in a point \cite[Propositions 3.6, 4.6]{vLW}. A better first step in answering Questions \ref{Q1} and~\ref{Q2} is therefore to determine which cliques to count, that is, we want to determine in which way 10 concurrent lines on a del Pezzo surface of degree 1 can be concurrent. From Lemma \ref{intersection multiplicity} we know that exceptional curves can intersect each other with multiplicities 1, 2, or 3. Moreover, from Remark \ref{rem: partners} we know how to distinguish between lines that are concurrent in a point on the ramification curve of the map (\ref{eq:map phi}), and lines that are concurrent in a point outside the ramification curve. In Theorem \ref{thm: 9 orbits} we list all cliques of size 10 that possibly correspond to concurrent exceptional curves. Moreover, we show that cliques of the same isomorphism type form one orbit under the action of the Weyl group on the set of these cliques. 

Let $G$ be the weighted graph on the 240 exceptional curves of a del Pezzo surface of degree 1, as in the previous section. Recall that the automorphism group of $G$ is the Weyl group $W_8$.

\begin{definition}
    We say that a clique $K$ in $G$ can be \textsl{realized as a generalized Eckardt point}, or simply \textsl{realized}, if there exists a del Pezzo surface of degree 1 with an isomorphism between the intersection graph on its 240 exceptional curves and $G$, such that $K$ corresponds to a set of exceptional curves concurrent in a generalized Eckardt point under this isomorphism. 
\end{definition}

Note that $W_8$ acts on the set of cliques of the same isomorphism type in $G$. The following lemma tells us that we only have to check one representative per orbit under this action in order to know if all cliques in the orbit can be realized. 

\begin{lemma}\label{lem: realizable iff conjuagte as well}
 For two cliques in $G$ that are conjugate under the action of the Weyl group, one can be realized as a generalized Eckardt point if and only if the other one can as well.
 \end{lemma}
 \begin{proof}Let $K_1=\{c_1,\ldots,c_n\}$ be a clique that can not be realized as a generalized Eckardt point, and let $K_2=\{w(c_1),\ldots,w(c_n)\}$ be another clique in $G$, which is the image of $K_1$ under the element $w\in W_8$. Assume by contradiction that $K_2$ can be realized as a generalized Eckardt point. Then there is a del Pezzo surface $X$ of degree 1, with an isomorphism $\psi$ between $G$ and the graph on the 240 exceptional curves on $X$, such that $\psi(K_2)$ corresponds to a set of exceptional curves that are concurrent on $X$.
 But then the map $\psi\circ w$ gives an isomorphism between $G$ and the graph of 240 exceptional curves on $X$, such that the image $\psi(w(K_1))=\psi(K_2)$ of $K_1$ under this isomorphism corresponds to a set of exceptional curves that are concurrent on $X$. This contradicts the fact that $K_1$ is not realizable as a generalized Eckardt point. We conclude that $K_1$ is realizable if and only if $K_2$ is. 
\end{proof}

Recall that we distinguish between generalized Eckardt points on a del Pezzo surface of degree~1 that are on the ramification curve of the map $(\ref{eq:map phi})$ and those that are outside the ramification curve, by only considering cliques in $G$ with edges of weights 1 and 3, or of weights 1 and 2, respectively. 

\begin{theorem}\label{thm: 9 orbits}
\begin{itemize}
\item[]
   \item[](i) There are 8 orbits of cliques of size 10 in $G$ under the action of the Weyl group with only edges of weights 1 and 2, and their isomorphism types are listed in Figure \ref{table cliques}. In characteristic 0, at most 6 of these (1, 2, 3, 5, 6, 8) correspond to cliques in $G$ that can be realized as generalized Eckardt points.
    \item[](ii) Outside characteristic 2, there is only one orbit of cliques of size 10 in $G$ under the action of the Weyl group with only edges of weights 1 and 3 that can be realized as a generalized Eckardt point. Cliques in this orbit are of the form depicted in Figure \ref{table clique 9}: a set of vertices $\{c_1,d_1,\ldots,c_5,d_5\}$, where $c_i\cdot c_j=1$ for $i\neq j$, and $c_i\cdot d_i=3$ for all $i$. 
\end{itemize} 
\end{theorem}
\begin{proof}
(i). Recall that cliques in $G$ with only edges of weights 1 and 2 correspond to cliques in $\Gamma_{\{-1,0\}}$. The table in \cite[Appendix A]{WvL} gives the complete list of orbits of the maximal cliques in $\Gamma_{\{-1,0\}}$, and we see that there are 6 orbits of maximal cliques of size 10. Translating this to cliques in $G$, a direct computation by hand shows that these are orbits $1-3$, $5-7$ in Figure \ref{table cliques}. We are left with determining the orbits of the non-maximal cliques of size 10 in $G$ with only edges of weights 1 and 2. By \cite[Proposition 19]{vLW}, there are no maximal cliques in $\Gamma_{\{-1,0\}}$ of size $11$, and there is one orbit of maximal cliques of size $12$ under the action of $W_8$, which is the maximal size of cliques in $\Gamma_{\{-1,0\}}$. When translating this to the graph $G$, a maximal clique of size 12 in this orbit has the form
\vspace{1em}
\begin{center}
\begingroup
\begin{tabular}{cccc}\label{four triangles}

 \raisebox{-3pt}{\begin{tikzpicture} [scale=0.8]
  
 \foreach \x /\alph  in {90/x, 210/y, 330/z}{
  \node[circle,fill,inner sep=0pt,minimum size=4pt,draw,xshift=-1cm] (\alph) at (\x:0.8cm) {}; }

 \foreach \x /\alph  in {90/a, 210/b, 330/c}{
  \node[circle,fill,inner sep=0pt,minimum size=4pt,draw,xshift=2cm] (\alph) at (\x:0.8cm) {}; }
  
   \foreach \x /\alph  in {90/d, 210/e, 330/f}{
  \node[circle,fill,inner sep=0pt,minimum size=4pt,draw,xshift=5cm] (\alph) at (\x:0.8cm) {}; }
  
  \foreach \x /\alph  in {90/g, 210/h, 330/i}{
  \node[circle,fill,inner sep=0pt,minimum size=4pt,draw,xshift=8cm] (\alph) at (\x:0.8cm) {}; }
  
  \path[every node/.style={font=\sffamily\small}]
  
     (x) edge node {} (y)
     (y) edge node {} (z)
     (z) edge node {} (x)
     (a) edge node {} (b)
     (a) edge node {} (c)
     (c) edge node {} (b)
     (d) edge node {} (e)
     (e) edge node {} (f)
     (f) edge node {} (d)
     (g) edge node {} (h)
     (h) edge node {} (i)
     (i) edge node {} (g);
    \end{tikzpicture}}

 \end{tabular}
 \captionof{figure}{Maximal clique $T$ of size twelve in $G$.}
\label{table cliques_2}
\endgroup
\end{center}
where edges drawn correspond to edges of weight 2 and every two vertices not connected by a drawn edge are connected by an edge of weight 1. Non-maximal cliques of size 10 in $G$ are contained in a clique of size 12. After removing two vertices from $T$ in Figure \ref{table cliques_2}, we get a clique of size 10 either isomorphic to clique 4 or clique 8 in Figure \ref{table cliques}. We are left to show that all cliques of type 4 in $G$ form one orbit under the action of $W_8$, and the same for type 8.

We obtain a clique of type 4 by removing two vertices from $T$ that are connected by an edge of weight 1. Let $H=\operatorname{Stab}_{W_8}(T)$ be the stabilizer of $T$ inside $W_8$. By \cite[Proposition 20 (ii)]{WvL} (after translating from the graph $\Gamma_{\{-1,0\}}$ to $G$), the group $H$ acts transitively on the set $$S_4=\{(e_{1}, e_{2},e_3,e_4) \in T^4 \mid e_{i} \cdot e_{j}=1\mbox{ for all }i\neq j\}.$$ Since we have a surjective map $S_4\longrightarrow S_2,$ with $$S_2=\{(e_{1}, e_{2}) \in T^2 \mid e_{1} \cdot e_{2}=1\},$$ we conclude that $H$ also acts transitively on $S_2$. But every clique of type 4 in $T$ is the complement of an element in $S_2$, so $H$ acts transitively on the set of cliques of type 4 in $T$ as well. Since $W_8$ acts transitively on the set of cliques of size 12 of isomorphism type $T$ in $G$, we conclude that $W_8$ acts transitively on the set of cliques of type 4 in $G$. 

We obtain a clique of type 8 by removing two vertices $e_1,e_2$ from $T$ that are connected by an edge of weight 2. By \cite[Corollary 15]{WvL}, the stabilizer $H$ of $T$ in $W_8$ acts transitively on the clique $T$. Let $e_3$ be the vertex in $T$ with $e_1\cdot e_3=e_2\cdot e_3=2$. Then for every other pair $\{c_1,c_2\}$ of vertices in $T$ intersecting with multiplicity 2, the element of $H$ that sends $e_3$ to the unique vertex $c_3$ with $c_1\cdot c_3=c_2\cdot c_3=2$ sends the pair $\{e_1,e_2\}$ to $\{c_1,c_2\}$. Therefore, $H$ acts transitively on the set 
$$S=\{\{e_1, e_2\} \mid e_1,e_2\in T,\;e_1\cdot e_2=2\}.$$ Since the cliques of type 8 in $T$ are complements of elements in $S$, we conclude that $H$ acts transitively of the cliques of type 8 in $T$, and, as before, we conclude that $W_8$ acts transitively on the set of cliques of type 8 in $G$. 

Now that we have proved that there are 8 orbits of cliques of size 10 in $G$ with only edges of weights 1 and 2, we are left to show that in characteristic 0, at most 6 of them can be realized as a generalized Eckardt point. In \cite[Proposition 4.6]{vLW} it is shown that a clique of type 4 can not be realized as a generalized Eckardt point outside characteristic 3, and the same is shown for a subset of a clique of type 7 in \cite[Proposition 5.4]{DW2} in characteristic 0. From Lemma \ref{lem: realizable iff conjuagte as well} we conclude that no cliques of type 4 or 7 can be realized as a generalized Eckardt point outside characteristic 3 and in characteristic 0, respectively. This concludes part (i).

(ii) Outside characteristic 2, the maximal size of a clique in $G$ with only edges of weights 1 and 3 corresponding to concurrent exceptional curves is 10 by \cite[Theorem 1.1]{vLW}. Moreover, from Remark \ref{rem: partners} it follows that such a clique should be of the form as claimed, since, for every element in such a clique, its partner intersects the same points on the ramification curve of $\varphi$. There is only one orbit of such cliques, which follows from \cite[Proposition 16]{WvL}. \qedhere \end{proof}

\vspace{11pt}

\begin{center}
\begingroup

\begin{tabular}{ccccccccc}\label{9 orbits}
\resizebox{4.5cm}{!}{\raisebox{1pt}{\begin{tikzpicture} [scale=0.3]
 \node [draw,circle,fill,inner sep=0pt,minimum size=2pt](t1) at (0,0) {};
 \node [draw,circle,fill,inner sep=0pt,minimum size=2pt](t2) at (1,0) {};
 \node [draw,circle,fill,inner sep=0pt,minimum size=2pt](t3) at (2,0) {};
 \node [draw,circle,fill,inner sep=0pt,minimum size=2pt](t3) at (3,0) {};
  \node [draw,circle,fill,inner sep=0pt,minimum size=2pt](t3) at (4,0) {};
 \node [draw,circle,fill,inner sep=0pt,minimum size=2pt](t3) at (2,-1) {}; 
 \node [draw,circle,fill,inner sep=0pt,minimum size=2pt](t3) at (2,-2) {}; 
 \foreach \x /\alph  in {90/a, 210/b, 330/c}{
  \node[circle,fill,inner sep=0pt,minimum size=2pt,draw,xshift=2cm,yshift=-0.2cm] (\alph) at (\x:0.8cm) {}; }
  \path[every node/.style={font=\sffamily\small}]
	(0,0) edge node {} (1,0)     
     (1,0) edge node {} (2,0)
     (2,0) edge node {} (3,0)
      (3,0) edge node {} (4,0)
      (2,0) edge node {} (2,-1)
      (2,-1) edge node {} (2,-2)
     (a) edge node {} (b)
     (a) edge node {} (c)
     (c) edge node {} (b);
    \end{tikzpicture}} }
 && &&
\resizebox{2cm}{!}{ \raisebox{1pt}{\begin{tikzpicture} [scale=0.3]
 \node [draw,circle,fill,inner sep=0pt,minimum size=2pt](t1) at (0,0) {};
 \node [draw,circle,fill,inner sep=0pt,minimum size=2pt](t2) at (1,0) {};
 \node [draw,circle,fill,inner sep=0pt,minimum size=2pt](t3) at (2,0) {};
 \node [draw,circle,fill,inner sep=0pt,minimum size=2pt](t3) at (3,0) {};
 \node [draw,circle,fill,inner sep=0pt,minimum size=2pt](t1) at (0,-2) {};
 \node [draw,circle,fill,inner sep=0pt,minimum size=2pt](t2) at (1,-2) {};
 \node [draw,circle,fill,inner sep=0pt,minimum size=2pt](t3) at (2,-2) {};
 \node [draw,circle,fill,inner sep=0pt,minimum size=2pt](t3) at (3,-2) {};
 \node [draw,circle,fill,inner sep=0pt,minimum size=2pt](t3) at (1,-1) {};
 \node [draw,circle,fill,inner sep=0pt,minimum size=2pt](t3) at (2,-1) {};
 \path[every node/.style={font=\sffamily\small}]
	(0,0) edge node {} (1,0)     
     (1,0) edge node {} (2,0)
     (2,0) edge node {} (3,0)
     (0,-2) edge node {} (1,-2)
     (1,-2) edge node {} (2,-2)
     (2,0) edge node {} (2,-1)
     (1,0) edge node {} (1,-1)
     (2,-2) edge node {} (3,-2);
 \end{tikzpicture}}}
 && &&
 \resizebox{3.3cm}{!}{
\raisebox{1pt}{\begin{tikzpicture} [scale=0.3]
 \node [draw,circle,fill,inner sep=0pt,minimum size=2pt](t2) at (1,0) {};
 \foreach \x /\alph  in {90/a, 150/b, 210/c,  270/d, 330/e, 30/f}{
  \node[circle,fill,inner sep=0pt,minimum size=2pt,draw,xshift=1cm] (\alph) at (\x:0.8cm) {}; }
    \foreach \x /\alph  in {90/g, 210/h, 330/i}{
  \node[circle,fill,inner sep=0pt,minimum size=2pt,draw,xshift=1.7cm] (\alph) at (\x:0.8cm) {}; }
 \path[every node/.style={font=\sffamily\small}]    
     (a) edge node {} (b)
     (b) edge node {} (c)
     (c) edge node {} (d)
     (d) edge node {} (e)
     (e) edge node {} (f)
     (f) edge node {} (a)
     (g) edge node {} (h)
     (h) edge node {} (i)
     (i) edge node {} (g);
 \end{tikzpicture}}}\\
1&&&&2&&&&3\\
&&&&&&&&\\
\resizebox{2.6cm}{!}{\raisebox{-3pt}{\begin{tikzpicture} [scale=0.3]
 \node [draw,circle,fill,inner sep=0pt,minimum size=2pt](t1) at (2.7,-1.5) {};
 \node [draw,circle,fill,inner sep=0pt,minimum size=2pt](t2) at (4.1,-1.5) {};
 \node [draw,circle,fill,inner sep=0pt,minimum size=2pt](t1) at (6,-1.5) {};
 \node [draw,circle,fill,inner sep=0pt,minimum size=2pt](t2) at (7.4,-1.5) {};
 \foreach \x /\alph  in {90/a, 210/b, 330/c}{
  \node[circle,fill,inner sep=0pt,minimum size=2pt,draw,xshift=2cm] (\alph) at (\x:0.8cm) {}; }
   \foreach \x /\alph  in {90/d, 210/e, 330/f}{
  \node[circle,fill,inner sep=0pt,minimum size=2pt,draw,xshift=1cm] (\alph) at (\x:0.8cm) {}; }
  \path[every node/.style={font=\sffamily\small}]
     (a) edge node {} (b)
     (a) edge node {} (c)
     (c) edge node {} (b)
     (d) edge node {} (e)
     (e) edge node {} (f)
     (f) edge node {} (d)
     (2.7,-1.5) edge node {} (4.1,-1.5)
     (6,-1.5) edge node {} (7.4,-1.5);
    \end{tikzpicture}}}
 &&&&
\resizebox{3cm}{!}{ \raisebox{1pt}{\begin{tikzpicture} [scale=0.3]
  \foreach \x /\alph  in {18/a, 90/b, 162/c,  234/d, 309/e}{
  \node[circle,fill,inner sep=0pt,minimum size=2pt,draw,xshift=1cm] (\alph) at (\x:0.8cm) {}; }
  \foreach \x /\alph  in {18/f, 90/g, 162/h,  234/i, 309/j}{
  \node[circle,fill,inner sep=0pt,minimum size=2pt,draw,xshift=1.8cm] (\alph) at (\x:0.8cm) {}; }
 \path[every node/.style={font=\sffamily\small}]   
     (b) edge node {} (c)
     (b) edge node {} (a)
     (a) edge node {} (e)
     (d) edge node {} (e)
     (c) edge node {} (d)
     (j) edge node {} (f)
     (f) edge node {} (g)
     (g) edge node {} (h)
     (h) edge node {} (i)
     (i) edge node {} (j);
      \end{tikzpicture}}}
&&&&
\resizebox{4cm}{!}{ \raisebox{1pt}{\begin{tikzpicture} [scale=0.3]
 \node [draw,circle,fill,inner sep=0pt,minimum size=2pt](t1) at (0,0) {};
 \node [draw,circle,fill,inner sep=0pt,minimum size=2pt](t1) at (1,0) {};
 \foreach \x /\alph  in {90/a, 0/b, 180/c,  270/d}{
  \node[circle,fill,inner sep=0pt,minimum size=2pt,draw,xshift=1cm] (\alph) at (\x:0.8cm) {}; }
  \foreach \x /\alph  in {90/e, 0/f, 180/g,  270/h}{
  \node[circle,fill,inner sep=0pt,minimum size=2pt,draw,xshift=1.75cm] (\alph) at (\x:0.8cm) {}; }
  \path[every node/.style={font=\sffamily\small}]
     (a) edge node {} (b)
     (a) edge node {} (c)
     (c) edge node {} (d)
     (d) edge node {} (b)
     (e) edge node {} (f)
     (f) edge node {} (h)
     (g) edge node {} (h)
     (g) edge node {} (e);
    \end{tikzpicture}}}\\
4&&&&5&&&&6\\
&&&&&&&&\\
\resizebox{3cm}{!}{\raisebox{1pt}{\begin{tikzpicture} [scale=0.3]
 \node [draw,circle,fill,inner sep=0pt,minimum size=2pt](t1) at (0,0) {};
 \node [draw,circle,fill,inner sep=0pt,minimum size=2pt](t2) at (1,0) {};
 \node [draw,circle,fill,inner sep=0pt,minimum size=2pt](t3) at (2,0) {};
 \node [draw,circle,fill,inner sep=0pt,minimum size=2pt](t3) at (3,0) {};
 \node [draw,circle,fill,inner sep=0pt,minimum size=2pt](t1) at (4,0) {};
 \node [draw,circle,fill,inner sep=0pt,minimum size=2pt](t2) at (5,0) {};
 \node [draw,circle,fill,inner sep=0pt,minimum size=2pt](t3) at (1,1) {};
 \node [draw,circle,fill,inner sep=0pt,minimum size=2pt](t3) at (1,-1) {};
 \node [draw,circle,fill,inner sep=0pt,minimum size=2pt](t3) at (4,1) {};
 \node [draw,circle,fill,inner sep=0pt,minimum size=2pt](t3) at (4,-1) {};
 \path[every node/.style={font=\sffamily\small}]
	(0,0) edge node {} (1,0)     
     (1,0) edge node {} (2,0)
     (1,0) edge node {} (1,-1)
     (1,0) edge node {} (1,1)
     (3,0) edge node {} (4,0)     
     (4,0) edge node {} (5,0)
     (4,0) edge node {} (4,-1)
     (4,0) edge node {} (4,1);
 \end{tikzpicture}}} 
&&&&
\resizebox{5.8cm}{!}{ \raisebox{-3pt}{\begin{tikzpicture} [scale=0.3]
 \node [draw,circle,fill,inner sep=0pt,minimum size=2pt](t1) at (0,0) {};
 \foreach \x /\alph  in {90/a, 210/b, 330/c}{
  \node[circle,fill,inner sep=0pt,minimum size=2pt,draw,xshift=1.75cm] (\alph) at (\x:0.8cm) {}; }
   \foreach \x /\alph  in {90/d, 210/e, 330/f}{
  \node[circle,fill,inner sep=0pt,minimum size=2pt,draw,xshift=0.75cm] (\alph) at (\x:0.8cm) {}; }
  \foreach \x /\alph  in {90/g, 210/h, 330/i}{
  \node[circle,fill,inner sep=0pt,minimum size=2pt,draw,xshift=2.75cm] (\alph) at (\x:0.8cm) {}; }
  \path[every node/.style={font=\sffamily\small}]
     (a) edge node {} (b)
     (a) edge node {} (c)
     (c) edge node {} (b)
     (d) edge node {} (e)
     (e) edge node {} (f)
     (f) edge node {} (d)
     (g) edge node {} (h)
     (h) edge node {} (i)
     (i) edge node {} (g);
    \end{tikzpicture}}}
&&&&\\
7&&&&8&&&&\\
 \end{tabular}
 
\captionof{figure}{The isomorphism types of 8 cliques in $G$ with only edges of weights 1 and 2. Each clique is a fully connected subgraphs of $G$ with the edges drawn corresponding to edges of weight~2. The edges not drawn correspond to edges of weight 1 in the clique.}
\label{table cliques}
\endgroup
\end{center}

\vspace{22pt}

\begin{center}
\begingroup
\begin{tabular}{cccc}

{\begin{tikzpicture} [scale=1.2]
  \node [draw,circle,fill,inner sep=0pt,minimum size=4pt](t1) at (0,0) {};
 \node [draw,circle,fill,inner sep=0pt,minimum size=4pt](t2) at (0,-1) {};
 \node [draw,circle,fill,inner sep=0pt,minimum size=4pt](t3) at (1,0) {};
 \node [draw,circle,fill,inner sep=0pt,minimum size=4pt](t4) at (1,-1) {};
 \node [draw,circle,fill,inner sep=0pt,minimum size=4pt](t5) at (2,0) {};
 \node [draw,circle,fill,inner sep=0pt,minimum size=4pt](t6) at (2,-1) {};
 \node [draw,circle,fill,inner sep=0pt,minimum size=4pt](t7) at (3,0) {};
 \node [draw,circle,fill,inner sep=0pt,minimum size=4pt](t8) at (3,-1) {};\node [draw,circle,fill,inner sep=0pt,minimum size=4pt](t9) at (4,0) {};
 \node [draw,circle,fill,inner sep=0pt,minimum size=4pt](t10) at (4,-1) {};
  \path[every node/.style={font=\sffamily\small}]
     (t1) edge node {} (t2)
     (t3) edge node {} (t4)
     (t5) edge node {} (t6)
     (t7) edge node {} (t8)
     (t9) edge node {} (t10);
     \end{tikzpicture}}

 \end{tabular}
 \captionof{figure}{The isomorphism type of a clique of size 10 in $G$ with only edges of weights 1 and 3. The clique is a fully connected subgraph of $G$ with the edges drawn corresponding to edges of weight~3. The edges not drawn correspond to edges of weight 1 in the clique.}
\label{table clique 9}
\endgroup
\end{center}

\begin{remark}\label{rem: on the ramification curve}
   In characteristic 2, the maximal number of concurrent exceptional curves in a point on the ramification curve of the map $(\ref{eq:map phi})$ is 16, and a clique in $G$ corresponding to such curves is of the form $K=\{c_1,d_1,\ldots,c_8,d_8\}$, where $c_i\cdot c_j=1$ for $i\neq j$, and $c_i\cdot d_i=3$ for all $i$. Moreover, there are no maximal cliques in $G$ with only edges of weights 1 and 3 of size smaller than 16, so any clique corresponding to 10 exceptional curves that are concurrent on the ramification curve is a subgraph of $K$. We see that there are 4 different isomorphism types, depending on how many pairs $c_i,d_i$ we remove from $K$: a clique of size 10 is left with 2, 3, 4, or 5 such pairs. With \texttt{magma}, we compute the number of orbits of the set of all cliques of size 10 in $G$ with only edges of weights 1 and 3 under the action of $W_8$. Since every such clique contains two vertices connected by an edge of weight 1, and $W_8$ acts transitively on the set of pairs of exceptional curves intersecting with multiplicity~1 \cite[Proposition 8]{WvL}, we can reduce computations and still find representatives for each orbit, by fixing two exceptional curves $e_1,e_2$ with $e_1\cdot e_2=1$ and computing all cliques of size 10 with edges of weights 1 and 3 containing $e_1$ and $e_2$. Interestingly enough, in this case there are multiple orbits with the same isomorphism type: those cliques of size 10 with 3 pairs $c_i,d_i$ as well as those with 4 such pairs split into two orbits. The other isomorphism types form a full orbit each. This gives a total of 6 orbits of cliques of size 10 with only edges of weights 1 and 3 in $G$. Note that this is also true outside characteristic 2, but in that case we already know that only one of these orbits contains cliques that can be realized as a generalized Eckardt point (the cliques consisting of 5 pairs, as described in Theorem \ref{thm: 9 orbits} (ii)). 
\end{remark}

\begin{corollary}\label{cor:7 possible cliques}
    If 10 lines on a del Pezzo surface of degree 1 over a field of characteristic 0 are concurrent in a point outside the ramification curve of the map $(\ref{eq:map phi})$, their intersection graph equals one of 1, 2, 3, 5, 6, 8 in Figure \ref{table cliques}.
\end{corollary}

To our knowledge there is no example in the literature of 10 concurrent exceptional curves on a del Pezzo surface of degree 1 whose intersection graph is isomorphic to one of cliques 1, 2, 3, 6, or 8 in Figure \ref{table cliques}. In Section \ref{sec:strategies} we discuss strategies for finding such examples, or for proving that more cliques are not realizable. 

\vspace{11pt}

We finish this section with a brief discussion on the relation between Question \ref{Q1} and the maximal number of generalized Eckardt points on del Pezzo surfaces of degree 2, a connection that was suggested to us by Bianca Viray. Over an algebraically closed field, every del Pezzo surface of degree 1 is the blow-up of a (non-unique) del Pezzo surface of degree 2, so another step in determining an upper bound for the number of generalized Eckardt points on a del Pezzo surface of degree 1 is to use what we know about del Pezzo surfaces of degree 2. A first question to answer is, which of the cliques in Figure \ref{table cliques} correspond to exceptional curves on a del Pezzo surface of degree 1 that, after blowing down an exceptional curve, could form a generalized Eckardt point on a del Pezzo surface of degree 2? In the following proposition we give a partial answer to this question. 

\begin{proposition}\label{prop: dP2's}Let $e_1,\ldots,e_{10}$ be 10 exceptional curves on a del Pezzo surface $X$ of degree~1. Assume there is another exceptional curve $c$ on $X$ such that, after blowing down $c$, there is a subset of 4 exceptional curves among $e_1,\ldots,e_{10}$ whose images on the resulting del Pezzo surface $Y$ of degree 2 are concurrent in a point. Then the intersection graph of $e_1,\ldots,e_{10}$ is not isomorphic to clique 4, 5, or 8 in Figure \ref{table cliques}.
\end{proposition}
\begin{proof}
    Assume that the set $\{e_1,\ldots,e_{10}\}$ contains a subset of 4 curves $e_{i_1},e_{i_2},e_{i_3},e_{i_4}$ that correspond to 4 concurrent exceptional curves on $Y$ after blowing down $c$. Since the exceptional curve on $Y$ are exactly the images of the 56 exceptional curves on $X$ that are disjoint from $c$, it follows that $e_{i_1},e_{i_2},e_{i_3},e_{i_4}$ are all disjoint from $c$. Choose a basis $[L],[E_1],\ldots,[E_8]$ for the exceptional classes in Pic $X$ as in Section \ref{sec:background}. We can write every exceptional curve as a vector $(a,b_1,\ldots,b_8)$ with respect to this basis as denoted in the discussion below Theorem \ref{exceptional curves}. Using this notation, observe that the set $S$ of 10 exceptional curves $$\{(1, 0, 0, 0, 0, 0, 0, 1, 1 ), ( 1, 0, 0, 1, 1, 0, 0, 0, 0 ), ( 1, 1, 1, 0, 0, 0, 0, 0, 0 ), ( 1, 0, 0, 0, 0, 1, 1, 0, 0 ),$$ 
    $$( 4, 2, 1, 2, 1, 2, 1, 1, 1 ), ( 4, 1, 2, 2, 1, 1, 1, 1, 2 ),( 4, 1, 1, 2, 1, 1, 2, 2, 1 ), (4, 1, 2, 1, 2, 1, 2, 1, 1 ),$$ 
    $$( 4, 2, 1, 1, 2, 1, 1, 2, 1 ), ( 4, 2, 1, 1, 1, 1, 2, 1, 2 )\}$$ 
    has intersection graph isomorphic to graph 4 in Figure \ref{table cliques}. It is a quick check with \texttt{magma} that no exceptional curve on $X$ is disjoint to any 4 curves in $S$ simultaneously, so no 4 curves in $S$ will correspond to concurrent exceptional curves after blowing down $c$. Since every clique in $G$ with isomorphism type that of clique 4 in Figure \ref{table cliques} is conjugate under the action of $W_8$ to the clique $S$ by Theorem \ref{thm: 9 orbits}, and the action of $W_8$ respects the intersection multiplicities, we conclude that the intersection graph of $\{e_1,\ldots,e_{10}\}$ is not isomorphic to graph 4 in Figure \ref{table cliques}. Cliques 5 and 8 are excluded analogously. 
\end{proof}
\begin{remark}
    From Proposition \ref{prop: dP2's} it follows that concurrent exceptional curves on a del Pezzo surface of degree 1 whose intersection graph looks like graphs 4, 5, or 8 in Figure \ref{table cliques} do not come from generalized Eckardt points on a del Pezzo surface of degree 2 (by blowing up a point). Therefore, we can not deduce an upper bound for the number of generalized Eckardt points realized by such cliques by using upper bounds for generalized Eckardt points on del Pezzo surfaces of degree 2. From Theorem~\ref{thm: 9 orbits} we know that clique 4 in Figure \ref{table cliques} can never be realized. We will show that clique 8 can be realized over $\mathbb{F}_{19}$ (Theorem \ref{thm: clique 8 in F19}). We also know that clique 5 can be realized, as mentioned in the introduction. It could be the case that, outside characteristic 3, a generalized Eckardt point on a del Pezzo surface of degree 1 outside the ramification curve can never come from a generalized Eckardt point on a del Pezzo surface of degree 2; for example, if it turns out that no clique except cliques 5 and 8 in Figure \ref{table cliques} can be realized as a generalized Eckardt point. 
\end{remark}

\section{Strategies - generalized Eckardt points outside the ramification curve}\label{sec:strategies}
From Theorem \ref{thm: 9 orbits} it follows that, outside characteristics 2 and 3 (where a generalized Eckardt point corresponds to 16 and 12 concurrent exceptional curves, respectively), we have a generalized Eckardt point outside the ramification curve of $(\ref{eq:map phi})$ if and only if it is possible to realize one of the cliques 1 -- 8 in Figure \ref{table cliques}, where we can exclude clique 4, and clique 7 in characteristic 0. However, apart from clique 5, it is unclear if such realizations are possible. In this section, we provide strategies to produce examples of 10 concurrent lines corresponding to a given clique, or to show that no examples can exist. 

\begin{definition}
    By a \textsl{representative} of a clique in the graph $G$ we mean an element from its orbit under the action of the Weyl group $W_8$. 
\end{definition}

For each clique in Figure \ref{table cliques}, we can use a representative of the corresponding orbit, compute the size of the stabilizer of that clique in the group $W_8$ with \texttt{magma}, and divide the order of the group $W_8$ (which is 696,729,600) by this number to obtain the size of its orbit under the action of $W_8$. The results are in the columns `Stabilizer Size' and `Orbit Size' of Table \ref{orbit_size_table}. The sizes of these orbits are very large, and obtaining a full list of all representatives would take a very long time with \texttt{magma}. In order to still obtain a list with many different representatives, we compute for each clique a `sub-orbit' as follows. Since the maximal size of cliques in $G$ with only edges of weight 2 is 3 \cite[Lemma 7]{WvL}, every clique in Figure \ref{table cliques} contains two vertices with an edge of weight 1. We fix two such vertices and compute with \texttt{magma} a list of only those cliques in each orbit that contain these two vertices. The size of these sub-orbits are found in Table \ref{orbit_size_table}, and the database containing all elements in all sub-orbits is available at \cite{DFIWcode}. 

\begin{table}[!htp]

	\centering
	\begin{tabular}{c|c|c|c}
		Clique & Stabilizer Size & Orbit Size & Sub-Orbit Size\\
		\hline
		1 & 18 & 38,707,200 & 92,160\\
		2 & 32 & 21,772,800 & 51,840\\
		3 & 36 & 19,353,600 & 46,080\\
		4 & 72 & 9,676,800 & 23,680\\
		5 & 100 & 6,967,296 & 16,128\\
		6 & 128 & 5,443,200 & 13,320\\
		7 & 192 & 3,628,800 & 8,880\\
		8 & 324 & 2,150,400 & 5,120
	\end{tabular}
	\caption{The stabilizer size, orbit size, and sub-orbit size for each clique in Figure \ref{table cliques}.}\label{orbit_size_table}
\end{table}

Since there are many representatives in each sub-orbit to choose from, a first step is determining which representatives will be best to work with. As we will discuss, there is not one best option, and we will need to consider trade-offs between them.

\vspace{11pt}

Fix a field $k$, and suppose we have points $P_1, P_2, \ldots, P_8$ in $\mathbb{P}_k^2$ in general position. For a del Pezzo surface $X$ of degree 1 obtained by blowing up these points, we denote representatives of cliques in the intersection graph $G$ on exceptional curves on $X$ by sets of curves in $\mathbb{P}^2$, using the correspondence in Theorem~\ref{exceptional curves}. We will use the following notation to simplify our discussion, letting $i,j,k$ be distinct indices in $\{1, 2, \ldots, 8\}$.

\begin{itemize}
	\item Let $\ell_{i,j}$ denote the line through points $P_{i}$ and $P_{j}$.
	\item Let $b_{i,j,k}$ denote the conic through points $P_m$ for $m \in \{1, \ldots, 8\}\setminus\{i,j,k\}$.
	\item Let $c_{i,j}$ denote the cubic through points $P_m$ for $m \in \{1,2,\ldots, 8\} \setminus \{i\}$ that is singular at point $P_{j}$.
	\item Let $q_{i,j,k}$ denote the quartic through points $P_1, P_2, \ldots, P_8$ that is singular at points $P_{i}, P_{j},$ and $P_{k}$. 
	\item Let $k_{i,j}$ denote the quintic through points $P_1, P_2, \ldots, P_8$ that is singular at all $P_m$ so that $m \in \{1,2,\ldots, 8\} \setminus \{i,j\}$.
	\item Let $s_{i}$ denote the sextic through points $P_1, P_2, \ldots, P_8$ that passes through $P_i$ three times and every other point twice.
\end{itemize}

\subsection{Clique Representatives}

First, it is always possible to pick a representative for cliques 1--8 containing four lines.

\begin{proposition}\label{four_line_prop}
	For each clique 1--8, there exists a representative containing the lines $\ell_{1,2}, \ell_{3,4}, \ell_{5,6},$ and $\ell_{7,8}$.
\end{proposition}
\begin{proof}
 We find such a representative for each clique in the list of sub-orbits produced with \texttt{magma}, available at \cite{DFIWcode}.
\end{proof}

\begin{example}
	Clique 1 has several representatives containing these four lines. We give two examples of such representatives below.\\
	
	Representative 1 has the above four lines together with the following exceptional curves
\begin{itemize}
	\item the cubic $c_{2,1}$,
	\item the quartics $q_{2,3,5}, q_{2,4,7}, q_{2,6,8},q_{3,6,7},  q_{4,5,8}.$
\end{itemize}

	Representative 2 has the above four lines together with the following exceptional curves
	\begin{itemize}
		\item the quartics $q_{1,3,5}, q_{1,4,7}, q_{1,6,8},q_{3,6,7}, q_{4,5,8}$,
		\item the sextic $s_2.$
	\end{itemize}
\end{example}

In this case, up to projective equivalence, we can fix four points in general position, say $P_1 = (0:1:1), P_3 =(1:0:1), P_5 = (1:1:1), P = (0:0:1)$, where $P$ is the point we wish to be concurrent to all 10 exceptional curves. Observe that this determines the following lines
\begin{align*}
	\ell_{1,2}&: x=0,\\
	\ell_{3,4}&: y=0,\\
	\ell_{5,6}&: x=y.
\end{align*}

Since $\ell_{7,8}$ contains $P$ and is distinct from the other three lines, then the line determined by $\ell_{7,8}$ must be of the form $x=gy$ for some $g \in k \setminus \{0\}$. Thus we can fix the following nine points for \textbf{point set-up A}

\begin{align*}
	P_1 &=(0:1:1), &  P_5 &=(1:1:1),\\
	P_2 &=(0:1:a), &  P_6 &=(1:1:c),\\
	P_3 &=(1:0:1), &  P_7 &=(d:1:e),\\
	P_4 &=(1:0:b), &  P_8 &=(d:1:f),\\
	P&=(0:0:1), & &
\end{align*}

where $a,b,c,d,e,f\in k$. We construct a del Pezzo surface of degree 1 by choosing $P_1, P_2 , \ldots , P_8$ in general position, and then blowing up the points.

An example of a del Pezzo surface of degree 1 with a generalized Eckardt point can be constructed by finding eight points in general position that pass through a given set of 10 concurrent exceptional curves and then blowing up these points. In order to determine whether such an example exists, we must determine whether it is possible to find eight points in general position of the above form so that these points determine 10 exceptional curves that pass through $P$. Since these eight points lie on lines $\ell_{1,2}, \ell_{3,4}, \ell_{5,6}$, and $\ell_{7,8}$ by construction, which are concurrent in $P$, we only need to check 6 more exceptional curves.

Notice that in this example there are several high-degree curves that must pass through the same point. As we will see later in this section, it may be computationally infeasible to compute such examples.

As an alternative, we search for representatives with fewer lines, but with more curves of low degree.

\begin{proposition}
	For each clique 1--8, there exists a representative containing the lines $\ell_{1,2}$ and $\ell_{3,4}$ and no other $\ell_{i,j}$.
\end{proposition}
\begin{proof}
    Every representative in the list of sub-orbits contains lines $\ell_{1,2}$ and lines $\ell_{3,4}$ by construction. By going through the list of sub-orbits we find a representative without any other line for each clique.
\end{proof}

\begin{example}
	Consider clique 1. We have the following examples of representatives with two lines.\\
	
	Representative 1 is described by the two lines above together with
	\begin{itemize}
		\item the conics $b_{1,4,7},b_{2,3,7}, b_{2,4,8}$
		\item the cubic $c_{5,8}, c_{6,7}$
		\item the quartics $ q_{2,4,5}, q_{2,6,7}$
		\item the sextic $s_{3}$.
	\end{itemize}

Representative 2 is described by the two lines above together with
\begin{itemize}
	\item the conics $b_{1,3,4},b_{2,4,8}$
	\item the cubics $  c_{2,8},c_{3,4}, c_{5,6}, c_{6,7},c_{7,5}$
	\item the sextic $ s_2$.
\end{itemize}
	
\end{example}

As in the previous case, we can fix four points $P_1 = (0:1:1), P_3 = (1:0:1), P_5 = (1:1:1),$ and $P = (0:0:1)$ in general position up to projective equivalence. We determine the lines
\begin{align*}
	\ell_{1,2}&: x=0,\\
	\ell_{3,4}&: y=0.
\end{align*}

Since we know that $P_5, P_6, P_7,$ and $P_8$ do not lie on lines $\ell_{1,2}$ and $\ell_{3,4}$, we know that their first coordinate cannot be equal to 0. Thus we can scale it to be 1. We then construct the points for \textbf{point set-up B}

\begin{align*}
	P_1 &= (0:1:1), & P_5 &=(1:1:1),\\
	P_2 &=(0:1:a), & P_6 &=(1:c:d),\\
	P_3 &=(1:0:1), & P_7 &=(1:e:f),\\
	P_4 &=(1:0:b), & P_8 &=(1:g:h),\\
	  P &=(0:0:1), & & 
\end{align*}
where $a,b,c,d,e,f,g,h \in k$ and $P$ is the point of intersection of lines $\ell_{1,2}$ and $\ell_{3,4}$. As before, we can find examples of generalized Eckardt points by constructing eight points in general position so that their blow-up produces a del Pezzo surface of degree 1 with 10 concurrent lines. Since we have already determined that lines $\ell_{1,2}$ and $\ell_{3,4}$ pass through $P$, we need to produce eight more exceptional curves that pass through $P$.

\subsection{Strategy for Computing Examples}
We now discuss a strategy for determining whether it is possible to have 10 concurrent lines on a del Pezzo surface of degree 1 in a fixed configuration. 

The first step is to determine an equation that must be satisfied for each exceptional curve in terms of the variables $a,b,\ldots, h$ used in the point set-up A or B. We illustrate the process using conics and cubics. The higher degree curves are similar. 

A general equation of a conic in $\mathbb{P}^2$ is given by $t_1 x^2 + t_2 xy+ t_3 xz + t_4y^2 + t_5 yz+ t_6 z^2.$ Observe that a conic $b_{i,j,k}$  is determined by five points of $P_1, P_2, \ldots, P_8$. As we want this conic to pass through $P$, we plug the points in $\{P_1, \ldots, P_8,P\} \setminus \{P_i, P_j, P_k\}$ into the general equation for the conic and obtain a system of six linear homogeneous equations in six unknowns. There is at least one $i$ with $t_i \ne 0$; by fixing such an $i$, we can rewrite the system of equations to obtain a single equation of the form $F t_i=0$ where $F \in k(a,b,c,d,e,f,g,h)$. Since $t_i \ne 0$, then $F=0$.

Observe that $F$ is a rational function, so our steps may be invalid if the denominator is 0. If the goal is to find an example of a generalized Eckardt point, we can find solutions to the numerator of $F$ equals 0 first, then check that the solutions we found make the denominator non-zero. Thus in order to find examples, we can ignore the denominator at first, then perform a simple check at the end. In order to show that no examples of a generalized Eckardt point realized by a given clique exist, the denominator must be considered.

The process for cubics is similar. A general equation of a cubic in $\mathbb{P}^2$ is given by $t_1x^3 + t_2x^2y+ t_3x^2z+ t_4xy^2 + t_5xz^2 + t_6xyz + t_7 y^3 + t_8 y^2z + t_9 yz^2 + t_{10} z^3.$ The cubic $c_{i,j}$ passes through $\{P_1, P_2, \ldots, P_8\} \setminus \{P_{i}\}$ and passes through $P_{j}$ twice. Since we want the cubic to pass through $P$, we first get eight equations from forcing each of the eight points to lie on the cubic. Since the cubic is singular at $P_{j}$, we construct two more equations using two of the three partial derivatives such that the remaining coordinate of the point is non-zero. The third partial derivative will be equal to 0 if we force the other two partials to be 0.
This gives us 10 linear homogeneous equations into 10 unknowns. We then solve as we did in the conic case. 

We can construct a single rational homogeneous equation for each exceptional curve in a similar way. Once we have an equation for each curve, we need to understand whether the variety formed by the numerators of these rational functions contains any $k$-points. The rational points of the corresponding variety are possible examples of 10 concurrent lines provided that no denominator is equal to 0. However, the last step is to ensure that the 8 points determined by the variety are in general position. If they are, we have produced an example of a del Pezzo surface of degree 1 with a generalized Eckardt point.

\begin{alg}\label{alg:strategy}
	\begin{enumerate}
 \item[]
		\item Input point set-up A or B and a representative of a clique.
		\item For each exceptional curve in the clique, construct a general equation for such a curve along with a square system of homogeneous equations that determine the coefficients of this equation.
		\item For each exceptional curve in the clique, what remains is a single rational function equal to 0. Construct a variety over $k$ from the numerators of these equations.
		\item Compute $k$-points on this variety and then find which rational points produce non-zero denominators for each rational function from Step 3. 
  \item For each remaining rational point from Step 4, check that the corresponding points $P_1, \ldots, P_8$ are in general position.
		\item If an example of eight points in general position exists, output the corresponding del Pezzo surface of degree 1.
	\end{enumerate}
\end{alg}

\begin{example}\label{ex:conic}
	Using point set-up A, consider the conic $b_{2,4,8}$. By forcing points $P_1, P_3, P_5, P_6, P_7,$ and $P$ to lie on this conic, we arrive at the equations
	\begin{align*}
		t_4 + t_5 + t_6 &=0\\
		t_1 + t_3 + t_6 &= 0\\
		t_1 + t_2 + t_3 + t_4 + t_5 + t_6 &=0\\
		t_1 + t_2 + ct_3 + t_4 + ct_5 + c^2t_6&=0\\
		d^2t_1 + dt_2 + det_3 + t_4 + et_5 + e^2t_6 &=0\\
		t_6 &=0.
	\end{align*}
	
	Substituting $t_6 =0, t_4 = -t_5, t_1 = -t_3$, we arrive at the simpler system
	\begin{align*}
	t_2&=0\\
	(c - 1)t_3 + (c - 1)t_5 &=0\\
	 (de - d^2)t_3 + (e - 1)t_5 &=0.
	\end{align*}

We solve $t_5 = -t_3$ and arrive at the single equation
$$
(de-d^2-e+1)t_3 =0.
$$

Since $t_3 \ne 0$, in order for all six points to lie on this conic, we need $de-d^2-e+1=0$. Notice that in order to solve for $t_5$ in the last step, we divided by $c-1$. This solution could be incorrect if $c=1$, but in this case $c=1$ is impossible since the points $P_1, \ldots, P_8$ should be in general position.

Since this equation is linear in $e$, one simplification we could make it to solve is $$e = \frac{d^2-1}{d-1} = d+1.$$ This is valid as long as $d \ne 1$.
\end{example}

There are two main points to make about Algorithm \ref{alg:strategy}. First, since we are ignoring denominators, showing that the variety has no $k$-points does not automatically mean that there cannot be an example of the given clique. Instead, more work must be done to show what happens if any of the denominators from Step 3 equal 0. 

The next main point is that even in \texttt{magma}, computing the variety formed from Algorithm \ref{alg:strategy} is very challenging. In fact, we have not been able to successfully compute the variety for any of the cliques 1--8. Thus in order to make progress toward this question, we need to make further simplifications.  

Example \ref{ex:conic} gives one possible simplification. Whenever an equation is linear in a variable, we could solve for that variable and make a substitution in the other exceptional curve equations. This reduces the dimension of the affine space over which the variety is defined at the cost of making the equations more complicated.

Another simplification is that we can fix more of the points in the beginning set-up. However, there is no guarantee that examples will exist in this special case, and we will not be able to prove rigorously that examples do not exist by using this method. 

A third simplification is to factor the polynomials that define the equations and use a single factor for each polynomial when defining the variety. This works if the polynomials factor, such as the polynomial in Example \ref{ex:conic} which factors as $$de - d^2-e+1 = (d-1)(e-d-1)$$ though many of the polynomials defined will be irreducible. Further, this splits the computation into several pieces, one for each combination of factors from each polynomial.

\begin{example}
	To illustrate how difficult computing these varieties can be, we write the equation that must be satisfied in order for points from point set B to lie on the cubic $c_{7,5}$, omitting the variable $t_i$ which we know must be nonzero. We have
	\begin{align*}
	0=&\big(-a b c^3 g h + a b c^3 g + a b c^3 h - a b c^3 + a b c^2 d g h - a b c^2 d g - a b c^2 d h + a b c^2 d \\
	&+ a b c^2 g h - a b c^2 g - a b c^2 h + a b c^2 +a b c d g^3 - a b c d g^2 h - a b c d g^2+ a b c d g \\
	& + a b c d h - a b c d - a b c g^3 + a b c g^2 h + ab c g^2 - a b c g h - a b d g^3 + a b d g^2 h\\
	&+ a b d g^2 - a b d g h + a b g^3 - a b g^2 h - a b g^2 + a b g h + a c^3 g h - a c^3 g - a c^3 h^2\\
	& + a c^3 h - a c^2 d g h + a c^2 d g + a c^2 d h^2 - a c^2 d h - a c^2 g h + a c^2 g + a c^2 h^2 \\
	&- a c^2 h - a c d g^3 + a c d g^2 h + a c d g^2 - a c d g - a c d h^2 + a c d h + a c g^3 - a c g^2 h \\
	&- a c g^2 + a c g h + a d^2 g^3 - a d^2 g^2 h - a d^2 g^2 + a d^2 g h - a d g^3 + a d g^2 h + a d g^2\\
	& - a d g h + b c^2 d g h - b c^2 d g - b c^2 d h + b c^2 d - b c^2 g h + b c^2 g + b c^2 h - b c^2 \\
	&- b c d^2 g h + b c d^2 g + b c d^2 h - b c d^2  -b c d g^2 h + b c d g^2 + b c d g h^2 - b c d g \\
	&-   b c d h + b c d + b c g^2 h - b c g^2 - b c g h^2 + b c g h + b d g^2 h - b d g^2 - b d g h^2 + b d g h \\
&- b g^2 h + b g^2 + b g h^2 - b g h - c^2 d g h + c^2 d g + c^2 d h^2 - c^2 d h + c^2 g h - c^2 g \\
	&- c^2 h^2 + c^2 h + c d^2 g h - c d^2 g - c d^2 h^2 + c d^2 h + c d g^2 h - c d g^2 - c d g h^2 + c d g \\
		&+ c d h^2 - c d h - c g^2 h + c g^2 + c g h^2 - c g h - d^2 g^2 h + d^2 g^2 + d^2 g h^2 - d^2 g h \\
		&+ d g^2 h - d g^2 - d g h^2 + d g h\big)/\big(b c d - b c - b d 
		+ b - c d + c + d^2 - d\big).
	\end{align*}

 The numerator of this rational function is irreducible.
\end{example}

\section{Surfaces with generalized Eckardt points on the ramification curve}\label{sec:families} 
As mentioned before, Example 5.2 in \cite{vLW} contains a del Pezzo surface of degree 1 over any field of characteristic unequal to 2,  3, 5, 7, 11, 13, 17, 19 with 10 concurrent exceptional curves on the ramification curve. In this section we extend this example to a larger family of surfaces with a generalized Eckardt point. For simplicity we work over a field of characteristic~0. 

\vspace{11pt}

Let $k$ be any field of characteristic $0$. For $a,b,c,d,e,f\in k $, define the following eight points in $\mathbb{P}_{k}^{2}$ as in point set-up A.

\begin{align*}
P_1 &=(0:1:1), & P_5 &=(1:1:1),\\
P_2 &=(0:1:a), & P_6 &=(1:1:c),\\
P_3 &=(1:0:1), & P_7 &=(d:1:e),\\
P_4 &=(1:0:b), & P_8 &=(d:1:f).
\end{align*}
It is easy to compute with \texttt{magma} the determinants of the matrices in \cite[Lemma 3.4]{vLW} that determine whether the points are in general position. This gives a set $S$ of polynomials in $k[a,b,c,d,e,f]$, and if none of these polynomials vanish the points $P_1, \ldots P_8$ are in general position. In that case, the blow-up of $\mathbb{P}_{k}^{2}$ in these points is a del Pezzo surface of degree 1, which we denote by~$X$. We define the following four lines in $\mathbb{P}^2$.
\begin{itemize}
	\item[] The line $\ell_{1}$ through $P_1$ and $P_2$, which is given by $x=0$,
	\item[] the line $\ell_{2}$ through $P_3$ and $P_4$, which is given by $y=0$,
	\item[] the line $\ell_{3}$ through $P_5$ and $P_6$, which is given by $x=y$,
	\item[] the line $\ell_{4}$ through $P_7$ and $P_8$, which is given by $x=dy$.
	
\end{itemize}
Note that $\ell_1, \ldots , \ell_4$ all contain the point $P=(0:0:1)$. Let $c_{7,8}$ be the unique cubic through $P_{1},P_{3}, \ldots ,P_6, P_{8}$ that is singular in $P_8$, and $c_{8,7}$ be the unique cubic through $P_{1}, \ldots , P_{7}$ that is singular in $P_7$.  
Define $e_{1}, \ldots, e_4$ to be the strict transforms on $X$ of $\ell_1, \ldots, \ell_4$, and $e_5$, $e_5'$ the strict transforms on $X$ of $c_{7,8}$ and $c_{8,7}$, respectively. These are exceptional curves on $X$. Now assume that both $c_{7,8}$ and $c_{8,7}$ contain $P$. Then the six exceptional curves $e_1, \ldots, e_5,e_5'$ are concurrent in a point $Q \in X$ that blows down to $P \in \mathbb{P}^2$. Let $\phi$ be the morphism $$ \phi: X \to \mathbb{P}^3$$ given by the linear system $|-2K_S|$ as in (\ref{eq:map phi}). By Theorem \ref{exceptional curves}, we can write $$e_5=3L - \sum_{i=1}^{6}E_i -2E_8 $$ and
 $$e_5'=3L - \sum_{i=1}^{6}E_i -2E_7,$$
 where $E_i$ are exceptional curves on $X$ above the point $P_i$ for $1 \le i \le 8$. It is easy to see that $e_5 \cdot e_5'=3$, therefore $Q$ is contained in the ramification curve of $\phi$ and for $1 \le i \le 4$, the partners of $e_i$ are concurrent in $Q$ by Remark \ref{rem: partners}, hence there are 10 exceptional curves on $X$ concurrent in $Q$. 

We now describe a family of del Pezzo surfaces of degree 1 in terms of the variables $a,b,c,d,e,f$ for which these curves are concurrent. The assumption that $c_{7,8}$ contains $P$ gives the polynomial equation $F_1\cdot s=0$, where $s$ is a product of elements in $S$, and 
\begin{align*}
	F_1(a,b,c,d,e,f)  =  &\;a(bd^2f - 2bdf + bf - bd^3 + bd^2 + bd - b - cf + c - df^2 + 2f^2
        + d^2f - 2f \\
     &   - d^2 + d)
      +(f-d)(bcd^2 - 2bdf + bf + bd - b - cdf - cf + cd + c + 2f^2 - 2f).
\end{align*}
We disregard the factor $s$, since we are interested in configurations of $P_1,\ldots,P_8$ in general position such that $c_{7,8}$ contains $P$. 
Write $$F_1(a,b,c,d,e,f)= a\cdot p(b,c,d,e,f)+q(b,c,d,e,f).$$ 
We claim that, if $P_1,\ldots,P_8$ are in general position and $F_1(a,b,c,d,e,f)=0$, then we have $ p(b,c,d,e,f) \ne 0$ and we can write $$a=- q(b,c,d,e,f)/p(b,c,d,e,f).$$ To prove this, assume that we have $p(b,c,d,e,f) = 0$. Since we assume $F_1(a,b,c,d,e,f)=0$, this implies $q(b,c,d,e,f)=0$. But then we have 
$$0=(f-d)p(b,c,d,e,f) + q(b,c,d,e,f)=d(d-f)(c +d- f - 1)(bd - f + 1),$$
and the latter four factors are all contained in $S$, as they correspond to configurations of $P_1,\ldots,P_8$ where $P_1,P_2,P_7,P_8$ are collinear, $P_3,P_5,P_8$ are collinear, $P_3,P_6,P_8$ are collinear, and $P_1,P_4,P_8$ are collinear, respectively. Therefore, if $P_1,\ldots,P_8$ are in general position we have $ p(b,c,d,e,f) \ne 0$.

\vspace{11pt}
Writing $a=- q(b,c,d,e,f)/p(b,c,d,e,f)$, we compute with \texttt{magma} the polynomial equation that corresponds to $c_{8,7}$ containing $P$, and we find that it factors as $F_2\cdot t=0$, where $t$ is a product of elements in $S$, and 
\begin{align*}
F_2 &= b^2cd^4 - 2b^2cd^3 + b^2cd^2 + 2b^2d^3ef - 5b^2d^2ef +
        4b^2def - b^2ef - 2b^2d^4e + 3b^2d^3e + b^2d^2e \\
&   -     3b^2de + b^2e - 2b^2d^4f + 3b^2d^3f + b^2d^2f - 3b^2df +
        b^2f + 2b^2d^5 - 3b^2d^4 + 2b^2d - b^2 \\
      &  + bc^2d^3 - bc^2d^2 +
        bcd^2ef - 3bcdef + 2bcef - 2bcd^3e + 2bcd^2e +
        2bcde - 2bce - 2bcd^3f \\
        &+ 2bcd^2f + 2bcdf - 2bcf +
        bcd^4 - 2bcd^2 - bcd + 2bc - 2bd^2e^2f      + 4bde^2 -
        2be^2f + 2bd^3e^2 \\
        &- 2bd^2e^2 - 2bde^2 + 2be^2 -
        2bd^2ef^2 + 4bdef^2 - 2bef^2 + 4bd^3ef - bd^2ef -
        7bdef + 4bef\\
        &- 2bd^4e - bd^3e + 4bd^2e + bde - 2be +
        2bd^3f^2 - 2bd^2f^2 - 2bdf^2 + 2bf^2 - 2bd^4f - bd^3f \\
        &+
        4bd^2f + bdf - 2bf + 3bd^4 - 3bd^3 - bd^2 + bd - c^2def
        - c^2ef + c^2de + c^2e + c^2df + c^2f \\
        &- c^2d - c^2 + 2ce^2f
        - 2ce^2 + 2cef^2 + cd^2ef - cdef - 4cef - cd^2e + cde
        + 2ce - 2cf^2 - cd^2f\\
        &+ cdf + 2cf + cd^2 - cd + 2de^2f^2
        - 4e^2f^2 - 2d^2e^2f + 4e^2f + 2d^2e^2 - 2de^2 - 2d^2ef^2 \\
        &+ 4ef^2 + 2d^3ef + 2d^2ef - 2def - 4ef - 2d^3e + 2de +
        2d^2f^2 - 2df^2 - 2d^3f + 2df + 2d^3 - 2d^2.
\end{align*}

Let $\mathbb{A}^5$ be affine space with coordinate ring $k[b,c,d,e,f]$. Points in $\mathbb{A}^5$ correspond to configurations of eight points $P_1,\ldots,P_8$ in $\mathbb{P}^2$. From what we have shown, it follows that all points in the algebraic set $V_1=Z(F_2)$ defined by $F_2=0$ that are not contained in the algebraic set $V_2=\cup_{s\in S}Z(s)$ correspond to a del Pezzo surface of degree 1 with a generalized Eckardt point on the ramification curve of the map (\ref{eq:map phi}). We know that there are points in the set $V_1\setminus V_2$, as \cite[Example 5.2]{vLW} gives an example: the point $Q=\left(-1, \tfrac54,-1, \tfrac12,-\tfrac12\right)$. We can extend this point to an infinite family as follows. 

\begin{theorem}\label{thm:family_on_ram_curve}
The intersection of $V_1$ with the hyperplanes $b=-1$, $c=\tfrac54$, and $d=-1$ gives an elliptic curve of rank 1 in $\mathbb{A}^5$, that contains infinitely many points corresponding to a del Pezzo surface of degree 1 with a generalized Eckardt point.
\end{theorem}
\begin{proof}
   With \texttt{magma} we check that $E$ is an elliptic curve of rank 1. Since it contains the point $Q$, which does not lie in $V_2$, we know that $E$ is not contained in $V_2$ (in fact, a quick point search with \texttt{magma} gives that all points on $E$ up to height 100 (which are 6) are not in $V_2$). Since $E$ is irreducible, we conclude that it can only intersect $V_2$ in finitely many points, hence there are infinitely many rational points on $E$ corresponding to a del Pezzo surface of degree 1 with a generalized Eckardt point.
\end{proof}

\section{Finite Fields}\label{sec:finite fields}
In this section we discuss the analogous problem of finding generalized Eckardt points over finite fields. When the characteristic of $k$ is 2 or 3, it is possible to have more than 10 concurrent lines. 

\begin{theorem}[{\cite{vLW}[Theorem 1.1]}]
	Let
	$P \in X(k)$
	be a point on the ramification curve of
	$\varphi$ in (\ref{eq:map phi}). The number of exceptional curves that go through $P$
	is at most ten if char $k \ne 2$, and at most sixteen if char $k =2$.
\end{theorem}

\begin{theorem}[{\cite{vLW}[Theorem 1.2]}]
	Let
	$P \in X(k)$
	be a point outside of the ramification curve of
	$\varphi$ in (\ref{eq:map phi}). The number of exceptional curves that go through $P$
	is at most ten if char $k \ne 3$, and at most twelve if char $k =3$.
\end{theorem}

Further, van Luijk and the fourth author produced examples of del Pezzo surfaces of degree 1 to show that the upper bounds in these theorems are sharp except possibly in char $k= 5$ outside the ramification curve. This leaves the following open question.

\begin{question}
Does there exist a del Pezzo surface of degree 1 in characteristic 5 with 10 concurrent lines outside the ramification curve?
\end{question}

In Theorem \ref{thm: 9 orbits}, we showed that there are exactly 8 possible configurations of 10 concurrent lines on a del Pezzo surface of degree 1 outside the ramification curve. 

\begin{question} \label{finite_field_ques}
	Which of the eight cliques can be realized over a finite field?
\end{question}

Answering this question will provide a full classification of all possible generalized Eckardt points over finite fields outside characteristics 2 and 3 (for the latter two characteristics it is known that there is one possible configuration for a generalized Eckardt point, and that these are realizable - see \cite{vLW}; generalized Eckardt points on the ramification curve outside characteristic 2 are described in Theorem \ref{thm: 9 orbits} (ii)).

\subsection{Results for $\mathbb{F}_p$, $p$ prime}
Let $p$ be a prime and consider the finite field $\mathbb{F}_p$. In order to search for examples of generalized Eckardt points, we first need to understand for which characteristics del Pezzo surfaces of degree 1 over $\mathbb{F}_p$ can exist.

\begin{theorem}\label{thm:finitefield17}
	For $p < 17$, there does not exist a set of eight points in general position in $\mathbb{P}_{\mathbb{F}_p}^2$ .
\end{theorem}

\begin{proof} We prove this theorem using \texttt{magma}. Up to projective equivalence, we can fix four points in general position. We then run Algorithm \ref{alg:strategy} to compute all possible collections of four additional points, and then determines whether these eight points are in general position. We find that for $p < 17$, there are no such eight points. 
\end{proof}

For $p \ge 17$, it makes sense to consider Question \ref{finite_field_ques}. We use point set-up A to construct the points 

\begin{align*}
	P_1 &=(0:1:1), & P_5&=(1:1:1),\\
	P_2 &=(0:1:a),& P_6 &=(1:1:c),\\
	P_3 &=(1:0:1), & P_7&=(d:1:e), \\
	P_4 &=(1:0:b), & P_8&=(d:1:f),\\
	P&=(0:0:1)
\end{align*}
 as in Section \ref{sec:strategies} and then use a brute force algorithm that (1) determines whether $P_1, P_2, \ldots, P_8$ are in general position and then (2) determines whether 10 lines pass through $P$.

 \begin{theorem}\label{thm: no Eckardt points primes 17,23}
	There are no examples of a generalized Eckardt point coming from cliques 1--8 over $\mathbb{F}_p$ for primes $p \in \{17,23\}$.
\end{theorem}

Interestingly, we do find examples of a generalized Eckardt point coming from clique 8 in $\mathbb{F}_{19}$.

\begin{theorem}\label{thm: clique 8 in F19}
Clique 8 can be realized over $\mathbb{F}_{19}$. 
    \end{theorem}
    \begin{proof}
        We use \texttt{magma} to run the brute force computation described above. We find several examples of 8 points in general position over $\mathbb{F}_{19}$ that pass through 10 concurrent exceptional curves. 
    \end{proof}

\begin{example}
    Our code fixes one representative of clique 8, with the exceptional curves 
    \begin{itemize}
        \item the lines $\ell_{12}, \ell_{34}, \ell_{56}, \ell_{78}$,
        \item the quartics $q_{135}, q_{238}, q_{246}, q_{147}, q_{257}, q_{168}$.
        \end{itemize}

    We find the the following points $P_1, \ldots, P_8$ are in general position and determine the 10 exceptional curves above. All 10 exceptional curves pass through $P$.
    \begin{align*}
        P_1 &= (0:1:1), & P_5&= (1:1:1),\\
        P_2 &= (0:1:2),  & P_6 &= (1:1:16),\\
        P_3 &= (1:0:1), & P_7 &= (7:1:18),\\
        P_4 &= (1:0:4), & P_8 &= (7:1:16),\\
        P &= (0:0:1).
    \end{align*}
        
        \end{example}

        \begin{remark}
            As shown in Table \ref{orbit_size_table}, there are many representatives of clique 8 that we can choose. Thus many more examples of realizations of clique 8 over $\mathbb{F}_{19}$ are likely to exist by changing the representative.
        \end{remark}
    
It seems difficult to push this further using a brute force method. The calculation took 1 hour for $p=17$, 2 hours for $p=19$, and overnight for $p=23$. 

\subsection{General $\mathbb{F}_q$}
Over $\mathbb{F}_{p}$, we showed that there are no points in general position for $p<17$ and we only found realizations of one clique in $\mathbb{F}_{19}$, but no realizations of any cliques when $p \in\{17, 23\}$. However, there are examples in the literature of generalized Eckardt points both on and outside the ramification curve in every characteristic (except characteristic 5 outside the ramification curve). Techniques to find such examples are similar to the technique described in the previous section; one must first choose a field extension of $\mathbb{F}_p$ by adding the roots of an irreducible polynomial. This is done in \cite[Section 5]{vLW}; for example, Example 5.1 shows 16 concurrent exceptional curves over a del Pezzo surface of degree 1 over $\mathbb{F}_{32}$, and Example 5.3 shows 10 concurrent exceptional curves on a del Pezzo surface of degree 1 on the ramification curve in characteristics 3, 5, 7, 11, 13, 17, and 19. One can see from this example that for $p= 5, 7, 11, 13, 17, 19$, examples already exist over a quadratic extension of $\mathbb{F}_p$, while in characteristic 3 it was needed to go to a cubic extension. Similarly, Example 5.4 exhibits 12 concurrent exceptional curves in a del Pezzo surface of degree~1 outside the ramification curve over $\mathbb{F}_{27}$. 

\section{Further questions}\label{sec: future}
In this paper, we make progress toward answering Questions \ref{Q1} and \ref{Q2}. We conclude with some questions that can help us make further partial progress.

It is known that cliques 4 and 7 cannot be realized outside characteristic 3 and in characteristic 0, respectively. Clique 5 can be realized at least outside characteristic 5, and the clique in Figure~\ref{table clique 9} can be realized in all characteristics; Theorem \ref{thm:family_on_ram_curve} gives a family of del Pezzo surfaces of degree 1 with such a generalized Eckardt point. However, it remains open if the remaining cliques can be realized over any field, with the exception of clique 8 over $\mathbb{F}_{19}$ (Theorem \ref{thm: clique 8 in F19}).

\begin{question}\label{ques:remaining_cliques}
    Is it possible to find a generalized Eckardt point that realizes clique 1, 2, 3, 6, or 8 over any field that is not $\mathbb{F}_{19}$?
\end{question}

While this paper focused on 10 concurrent lines over any field, 10 concurrent lines do not produce a generalized Eckardt point in characteristics 2 or 3. In these characteristics it is known that there is only one configuration of maximal cliques of exceptional curves, which can be realized as a generalized Eckardt point. It might therefore be easier to answer Question \ref{Q1} for these cases. 

\begin{question}
    How many generalized Eckardt points can a del Pezzo surface of degree 1 over a field of characteristic 2 or 3 contain?
\end{question}

Finally, since all known examples give a single generalized Eckardt point, a more attainable question leading to Question \ref{Q1} would be the following. 

\begin{question}
  Does there exist a del Pezzo of degree 1 that has more than 1 generalized Eckardt point?
\end{question}

\bibliographystyle{amsalpha}
\bibliography{bibliography}
\end{document}